\documentclass[12pt]{amsart}

\usepackage{amssymb,amsthm}

\newtheorem{thm}{Theorem}
\newtheorem{theorem}{Theorem}[section]

\newtheorem{lemma}[theorem]{Lemma}

\begin{document}
	
\title[Tidy Groups]{Finite solvable tidy Groups whose orders are divisible by two primes}

\author[Beike]{Nicolas F. \ Beike}
\address{
	Department of Mathematical Sciences, Kent State University, Kent, OH 44242}
\email{nbeike@kent.edu}

\author[Carleton]{Rachel Carleton}
\address{Department of Mathematical Sciences, Kent State University, Kent, OH 44242}
\email{rcarlet3@kent.edu}

\author[Costanzo]{David G.\ Costanzo}
\address{School of Mathematical and Statistical Sciences, O-110 Martin Hall, Box 340975, Clemson University, Clemson, SC 29634}
\email{davidgcostanzo@gmail.com}

\author[Heath]{Colin Heath}
\address{
    New York University School of Law, 40 Washington Square South, New York, NY 10012}
\email{colin.heath@law.nyu.edu}

\author[Lewis]{Mark L.\ Lewis}
\address{Department of Mathematical Sciences, Kent State University, Kent, OH 44242}
\email{lewis@math.kent.edu}

\author[Lu]{Kaiwen Lu}
\address{
    Department of Mathematics, Brown University, Providence, RI 02912}
\email{kaiwen$\_$lu@brown.edu}

\author[Pearce]{Jamie D. Pearce}
\address{Department of Mathematics, University of Texas at Austin, 2515 Speedway, PMA 8.100, Austin, TX 78712}
\email{jamie.pearce@utexas.edu}

\keywords{tidy groups, $\{p, q\}$-groups, solvable groups}
\subjclass[2010]{Primary: 20D10 Secondary: 20D20 }

\begin{abstract}
In this paper, we investigate finite solvable tidy groups.  
We classify the tidy $\{ p, q \}$-groups.  Combining this with a previous result, we are able to characterize the finite tidy solvable groups.  Using this characterization, we bound the Fitting height of finite tidy solvable groups and we prove that the quotients of finite tidy solvable groups are tidy.   	
\end{abstract}

\maketitle


\section{Introduction}

Throughout this paper, all groups are finite except where stated.  We use \cite{Isagrp} for standard group theory results; but any standard group theory text will contain nearly all of the results that we need. For a group $G$ and an element $x \in G$, we condier the set ${\rm Cyc}_G (x) = \{ g \in G \mid \langle x, g \rangle {\rm~is~cyclic} \}$.  For most groups $G$ and elements $x$, this set ${\rm Cyc}_G (x)$ is not a subgroup. In this paper, we focus on the following special situation.  As in \cite{cycels}, a group $G$ is said to be {\it tidy} if ${\rm Cyc}_G (x)$ is a subgroup of $G$ for every element $x \in G$.  

Using the definition of tidy, it is an easy exercise to show that subgroups of tidy groups are tidy.  On the other hand, in \cite{cycels}, it was shown that if $G$ is nilpotent and the Sylow subgroups of $G$ are tidy, then $G$ is tidy.  In \cite{ourpre}, we determine that a $p$-group is tidy if and only if it is cyclic, has exponent $p$, is dihedral, or is generalized quaternion.  Therefore, we know all of the nilpotent tidy groups.

It is natural next to look at solvable groups.  First, one would ask if there are non-tidy solvable groups where all of the Sylow subgroups are tidy.  In \cite{pre2}, we show that $S_3 \times Z_3$ is an example of a solvable group where all of the Sylow subgroups are tidy but the group itself is not tidy.  However, when we consider Hall subgroups with two prime divisors, we obtain a different answer.  The main result in \cite{pre2} shows that if $G$ is a solvable group and every Hall subgroup whose order is divisible by two prime divisors is tidy, then $G$ is tidy.  

Thus, if we can classify the $\{ p, q \}$-groups that are tidy for primes $p$ and $q$, then we can use 
the main result from \cite{pre2} to characterize the solvable tidy groups.  In fact, we are able to classify the tidy $\{ p, q \}$-groups in Theorem \ref{Lemma 1.17.}.  Due to the complexity of the list of groups, we will leave the presentation of this classification for Section \ref{secntidypq}.  We now present two consequences for solvable tidy groups that we can obtain from our characterization.

First, we are able to bound the Fitting height of the group and the derived length of the quotient modulo the Fitting subgroup.

\begin{thm} \label{Lemma 1.15.} 
Let $G$ be a solvable, tidy group. Then $G$ has Fitting height at most $4$ and $G/F(G)$ has derived length at most $4$. If $|G|$ is odd, then $G$ has Fitting height at most $3$ and $G/F(G)$ is abelian or metabelian.
\end{thm}

We note that the definition of tidy groups is not particularly compatible with quotients. It is noted by Erfanian and Farrokhi in the Introduction of \cite{classes} that the quotients of infinite tidy groups are not necessarily tidy.  They give an example of an infinite tidy group with a quotient that is not tidy.

We now give some evidence that for finite tidy groups, quotients of tidy groups are tidy.  In particular, we prove that the quotients of finite solvable tidy groups are tidy.

\begin{thm} \label{quotient}
If $G$ is a (finite) solvable tidy group and $N$ is a normal subgroup of $G$, then $G/N$ is tidy.
\end{thm}

As the above shows, one of our methods of analyzing tidy groups is in terms of their Hall subgroups.  In order to obtain many of our results, we also need a second method of analysis.  This method involves studying the elements of prime power order.  We will show in Section \ref{secnprimepower} that it is sufficient to check whether the elements of prime power order satisfy the condition of the definition to determine if a group is tidy.  



This research was conducted during a summer REU in 2020 at Kent State University with the funding of NSF Grant DMS-1653002.  We thank the NSF and Professor Soprunova for their support.   

\section{Elements of Prime Power Order} \label{secnprimepower}

In \cite{pre2}, we prove the following result which shows that need only check on the elements of prime power order.  

\begin{lemma}\label{Lemma 1.22.}
Suppose $G$ is a group. If every element $1 \ne x \in G$ having prime power order satisfies the condition that ${\rm Cyc}_G (x)$ is a subgroup of $G$, then $G$ is a tidy group.  
\end{lemma}


With this in mind, we can focus on the elements of prime power order.  Since subgroups of tidy groups are tidy, we need to focus on tidy Sylow $p$-subgroups.   In Theorem 14 of \cite{cycels}, they obtain a characterization of the tidy $p$-groups.  In \cite{ourpre}, we use their characterization to obtain a classification of the tidy $p$-groups.

\begin{theorem}\label{p-groups}
Let $G$ be a $p$-group for some prime $p$.   Then $G$ is a tidy group if and only if one of the following occurs:.
\begin{enumerate}
\item $G$ has exponent $p$.
\item $G$ is cyclic.
\item $p = 2$ and $G$ is dihedral or generalized quaternion.	
\end{enumerate}
\end{theorem}

The following is a consequence of  Burnside's normal $p$-complement theorem (see Theorem 5.13 of \cite{Isagrp}) and Fitting's theorem (see Theorem 4.34 of \cite{Isagrp}).  It is proved as Lemma 2.5 of \cite{pre2}.

\begin{lemma} \label{tidy cyclic 1}
Let $G$ be a group and let $p$ be a prime.  If $G$ has a cyclic Sylow $p$-subgroup and $p$ divides $|Z(G)|$, then $G$ has a normal $p$-complement.
\end{lemma}

When a Sylow $p$-subgroup of $G$ is cyclic or generalized quaternion and $x \in Z(G)$ has $p$-power order, then ${\rm Cyc}_G (x)$ is a subgroup.  This was proved in Lemma 2.6 of \cite{pre2}.

\begin{lemma}\label{Lemma 1.19.}
Let $G$ be a group and let $p$ be a prime.  Suppose $1 \ne x \in Z(G)$ has $p$-power order. If a Sylow $p$-subgroup of $G$ is either cyclic or generalized quaternion, then $G = {\rm Cyc}_G (x)$.
\end{lemma}

For central elements of order $p$ when the Sylow subgroup has exponent $p$, we proved the following in Lemma 2.4 of \cite{pre2}. 

\begin{lemma} \label{tidy exp p 1}
Let $G$ be a group and let $p$ be a prime.  Suppose $x \in Z(G)$ has order $p$ and a Sylow $p$-subgroup of $G$ has exponent $p$.  Then ${\rm Cyc}_G (x)$ is a subgroup of $G$ if and only if $G$ has a normal $p$-complement.  In this case, ${\rm Cyc}_G (x) = \langle x \rangle K$ where $K$ is the normal $p$-complement of $G$.
\end{lemma}

For a central element of order $2$ in a group whose Sylow $2$-subgroups are dihedral and there is a normal $2$-complement, we proved the following as Lemma 2.7 of \cite{pre2}.  It  made use of a result proved by Gorenstein and Walter in \cite{GoWa} regarding groups that have a dihedral subgroup as a Sylow $2$-subgroup.

\begin{lemma}\label{Lemma 1.20.}
Let $G$ be a group and let $x \in Z(G)$ have order $2$. Suppose a Sylow $2$-subgroup $T$ of G is dihedral and write $D$ for the cyclic subgroup of index $2$ in $T$. Then $G$ has a normal $2$-complement $K$ and ${\rm Cyc}_G (x) = DK$. In particular, ${\rm Cyc}_G (x)$ is a subgroup of G.
\end{lemma} 
The following lemma is key to understanding how coprime subgroups interact in a tidy group. 

\begin{lemma} \label{Lemma 1.6.} 
Let $G$ be a tidy group, and let $p$ be a prime. If $O_p(G) > 1$ and $x \in G$ has order not divisible by $p$, then one of the following occurs:
\begin{enumerate}
\item $x$ centralizes $O_p (G)$.
\item $\langle x \rangle O_p (G)/C_{\langle x \rangle} (O_p(G))$ is a Frobenius group with Frobenius kernel $O_p (G) \cong O_p (G) C_{\langle x \rangle} (O_p(G))/C_{\langle x \rangle} (O_p (G))$.
\item $p = 2$, the subgroup $O_2 (G)$ is the quaternion group of order $8$, and $\langle x \rangle O_2(G)/C_{\langle x \rangle} (O_2(G))$ is isomorphic to ${\rm SL}_2 (3)$.
\end{enumerate}
\end{lemma}

\begin{proof} 
Without loss of generality, we may assume that $G = \langle x \rangle O_p(G)$.  If $x$ centralizes $O_p (G)$, then we are done. Thus, we may assume that $x$ does not centralize $O_p (G)$.  We know via Theorem \ref{p-groups} that a Sylow $p$-subgroup is either cyclic, dihedral, generalized quaternion, or has exponent $p$.  It is not difficult to see that this implies that $O_p (G)$ must also be in this same list.  If $O_p (G)$ is dihedral, generalized quaternion but not quaternion, or cyclic of $2$-power order, then we know that ${\rm Aut} (O_p (G))$ is a $2$-group and we conclude that $x$ must centralize $O_p (G)$.  Suppose that $O_2 (G)$ is the quaternion group of order $8$.  In this case, we know that ${\rm Aut} (O_2 (G))$ is isomorphic to $S_4$. This implies that $x C_{\langle x \rangle} (O_2(G))$ has order $3$ and $\langle x \rangle O_2(G)/C_{\langle x \rangle} (O_2 (G))$ is isomorphic to ${\rm SL}_2 (3)$. 

We now may assume that either $p$ is odd or $O_p (G)$ is elementary abelian if $p = 2$. Hence, either $O_p (G)$ is cyclic or has exponent $p$. Suppose that there exists $1 \ne z \in Z(O_p (G))$ with $C_{\langle x \rangle} (z) > C_{\langle x \rangle} (O_p(G))$.  By Lemmas \ref{tidy exp p 1} or \ref{tidy cyclic 1}, $C_G (z)$ has a normal $p$-complement. Observe that $O_p (G)$ is contained in $C_G (z)$. Now, $C_{\langle x \rangle} (z)$ and $O_p (G)$ are contained in normal subgroups of $C_G (z)$ with coprime orders. This implies that $C_{\langle x \rangle} (z)$ centralizes $O_p (G)$ which is a contradiction. Thus, we can conclude that $C_{\langle x \rangle} (z) = C_{\langle x \rangle} (O_p (G))$ for all $z \in Z (O_p (G))$. (As $z \in O_p (G)$, the inclusion $C_{\langle x \rangle} (O_p (G)) \le C_{\langle x \rangle} (z)$ holds, and we have just shown that the containment cannot be proper.)  

Now, suppose $1 \ne y \in O_p (G)$.  We see that $Z(O_p (G)) \leq C_G(y)$.  By Lemmas \ref{tidy exp p 1} or \ref{tidy cyclic 1} again, we see that $C_G (y)$ has a normal $p$-complement. This implies that $C_{\langle x \rangle} (y)$ must centralize $Z (O_p (G))$. Thus, $C_{\langle x \rangle} (y) \leq C_{\langle x \rangle} (z) \leq C_{\langle x \rangle} (O_p (G))$ for all $1 \ne y \in O_p(G)$.   We deduce that the group  $\langle x \rangle O_p (G)/C_{\langle x \rangle} (O_p(G))$ is a Frobenius group with Frobenius kernel $O_p (G) \cong O_p (G) C_{\langle x \rangle} (O_p (G))/C_{\langle x \rangle} (O_p (G))$, and the result is proved.
\end{proof} 

We will see that we need to understand the centralizer of $O_p (G)$ in a Hall $p$-complement.  In this next lemma, we consider this situation when a Sylow $p$-subgroup has exponent $p$.

\begin{lemma} \label{tidy exp p 2}
Let $G$ be a tidy group and let $p$ be a prime.   Suppose that $G$ has a Sylow $p$-subgroup $P$ of exponent $p$ and let $H$ be a Hall $p$-complement of $G$.  If $O_p (G) > 1$ and $C = C_H (O_p(G))$, then $C$ is normal in $G$.  In addition, if $C < H$, then $O_p (G)H/C$ is a Frobenius group with Frobenius kernel $O_p (G) \cong O_p (G)C/C$ and Frobenius complement $H/C$ and $P \cap Z(G) = 1$.
\end{lemma}

\begin{proof}
By order considerations, $G = PH$.  Since $1 < O_p (G)$ is a normal subgroup of $P$, the intersection $O_p (G) \cap Z(P) > 1$.  Thus, we can fix an element $z \in Z(P) \cap O_p (G)$.  Observe that $C$ and $P$ are both contained in $C_G (z)$.  By Lemma \ref{tidy exp p 1}, we see that $C_G (z)$ has a normal $p$-complement.  Observe that $C_G (O_P (G)) \le C_G (z)$, and so, $C_G (O_p (G))$ has a normal $p$-complement.  Since $O_p (G)$ is normal in $G$, we see that $C_G (O_p (G))$ is normal.  Thus, $C = H \cap C_G (O_p (G))$ is a Hall $p$-complement in $C_G (O_p (G))$.  It follows that $C$ is normal in $C_G (O_p(G))$, and thus, $C$ is characteristic in $C_G (O_p(G))$.  Using the fact that $C_G (O_p (G))$ is normal in $G$, we conclude that $C$ is normal in $G$.

We now assume that $C < H$.  By Lemma \ref{Lemma 1.6.}, we see that every element in $H \setminus C$ acts Frobeniusly on $O_p (G)$.  It follows that $HO_p(G)/C$ is a Frobenius group with Frobenius kernel $O_p (G) \cong O_p (G)C/C$ and Frobenius complement $H/C$. Notice that $Z(G) \cap P$ is a normal $p$-subgroup of $G$, so it is contained in $O_p (G)$.  Since $O_p (G) \cong O_p (G)C/C$ and $O_p (G)C/C$ is a Frobenius kernel, we conclude that $Z(G) \cap P = 1$.  
\end{proof}

In this next lemma, we make use of a result proved by Gorenstein and Walter in \cite{GoWa} regarding groups that have a dihedral subgroup as a Sylow $2$-subgroup.  




\begin{lemma} \label{Lemma 1.7.} 
Let $G$ be a  group. Suppose a Sylow $2$-subgroup of $G$ is a dihedral group. If $O_2 (G) > 1$ and is not elementary abelian of order $4$, then $G$ has a normal $2$-complement.
\end{lemma}

\begin{proof}
Let $P$ be a Sylow $2$-subgroup of $G$. We know that $P$ is a dihedral group and $O_2 (G)$ is a normal subgroup of $P$. Since $O_2 (G)$ is not elementary abelian of order $4$, we see that $O_2 (G)$ is either cyclic or dihedral. In either case, $O_2 (G)$ has a subgroup $Z$ of order $2$ that is characteristic in $O_2 (G)$. It follows that $Z$ is normal in $G$. Hence, $G$ is the centralizer of an involution. Since a Sylow $2$-subgroup of $G$ is dihedral, we can apply Lemma 8 of \cite{GoWa} 
to see that $G$ has a normal $2$-complement.
\end{proof}

We continue to investigate tidy groups that have a Sylow $2$-subgroup that is dihedral.  While we are not assuming that $G$ is solvable, we do assume that $G$ has a Hall $2$-complement.

\begin{lemma} \label{Lemma 1.8.} 
Let $G$ be a tidy group. Suppose a Sylow $2$-subgroup of $G$ is a dihedral group. Let $H$ be a Hall $2$-complement of $G$ and let $C = C_H (O_2 (G))$. If $O_2(G) > 1$, then $C$ is normal in $G$ and either $G$ has a normal $2$-complement or $G/C  \cong S_4$.
\end{lemma}

\begin{proof} 
When $O_2 (G)$ is not elementary abelian of order $4$, we may use Lemma \ref{Lemma 1.7.} to see that $G$ has a normal $2$-complement. Suppose $O_2 (G)$ is elementary abelian of order $4$. Write $P$ for a Sylow $2$-subgroup of $G$.  Since $O_2 (G)$ is a normal subgroup of $P$, we deduce that $P$ is dihedral of order $8$ and $C_P (O_2 (G)) = O_2 (G)$.  In particular, $C_G (O_2 (G))$ has an elementary abelian Sylow $2$-subgroup. By Lemma \ref{tidy exp p 2}, we see that $C$ is a normal $2$-complement of $C_G (O_2(G))$.  We see that $G/C$ is isomorphic to a subgroup of ${\rm Aut} (O_2 (G)) \cong S_4$.  Since a Sylow $2$-subgroup of $G$ is dihedral, the result now follows. 
\end{proof}

\section{Some tidy groups}

In this section, we prove that several groups are tidy groups. 
The following generalizes Proposition 2.5 of \cite{ctidynote}.  That result proves that a Frobenius group is tidy if and only if the Frobenius kernel is tidy.  We now consider certain extensions of Frobenius groups to determine when they are tidy.

\begin{theorem} \label{Lemma 1.26.}
Let $G$ be a group and let $N < F(G)$ be a normal subgroup of $G$ so that $G/N$ is a Frobenius group with Frobenius kernel $F (G)/N$ and Frobenius complement $H/N$.  Assume that $N$ is a Hall subgroup of $F(G)$ and either (1) $H$ is nilpotent or (2) $N$ is a Hall subgroup of $H'N$, 
$H'N/N$ has a complement $U/N$ in $H/N$, and $H'$ and $U$ are nilpotent groups. If the Sylow $p$-subgroups of $G$ are tidy for every prime $p$ that divides $|F(G)|$, then $G$ is tidy.
\end{theorem}

\begin{proof} 
By Lemma \ref{Lemma 1.22.}, to prove that $G$ is tidy, it suffices to prove that ${\rm Cyc}_G (x)$ is a subgroup for every element $1 \ne x \in G$ that has prime power order.  Assume $x \in G$ has order that is a power of the prime $p$. 

Suppose first that $p$ does not divide $|F(G)|$.  Let $P$ be a Sylow $p$-subgroup of $G$.  We see that $P \cap N \le P \cap F(G) = 1$, so $P \cong PN/N$.  This implies that $P$ is isomorphic to a Sylow subgroup of a Frobenius complement.  Hence, $P$ is either cyclic or generalized quaternion. By Lemma \ref{Lemma 1.19.}, we see that ${\rm Cyc}_{C_G (x)} (x) = C_G(x)$. Since ${\rm Cyc}_G (x) \subseteq C_G(x)$, we deduce that ${\rm Cyc}_G (x) = C_G (x)$ is a subgroup of $G$.

We now suppose that $p$ divides $|F(G)|$.  We are assuming that $G$ has tidy Sylow subgroups for every prime dividing $|F (G)|$.  This implies that all of the Sylow subgroups of $F(G)$ are tidy.  Since $F (G)$ is nilpotent, we conclude that $F(G)$ is tidy.  We next suppose that $p$ does not divide $|N|$.  This implies that $xN$ is a nontrivial element of $F (G)/N$. Since $G/N$ is a Frobenius group, we deduce that $C_{G/N} (xN) \leq F (G)/N$.  It is not difficult to see that $C_G (x) N/N \le C_{G/N} (xN)$.  It then follows that $C_G (x) \leq F (G)$.  We now have that ${\rm Cyc}_G (x) = {\rm Cyc}_{F(G)} (x)$.  Since $F (G)$ is tidy, we conclude that ${\rm Cyc}_G (x)$ is a subgroup of $F(G)$ and hence, it is a subgroup of $G$.

We next consider that $p$ divides $|N|$.  If a Sylow $p$-subgroup of $C_G(x)$ is either cyclic or generalized quaternion, then we know that ${\rm Cyc}_G (x) = C_G (x)$ is a subgroup by Lemma \ref{Lemma 1.19.}. Thus, we may assume that a Sylow $p$-subgroup of $C_G (x)$ either has exponent $p$ or is dihedral.  In view of Lemmas \ref{tidy exp p 1} and \ref{Lemma 1.20.}, to show that ${\rm Cyc}_G (x)$ is a subgroup, it suffices to show that $C_G (x)$ has a normal $p$-complement. 


Let $\sigma$ be the set of prime divisors of $|F(G) : N|$, and let $L$ be the Hall $\sigma$-subgroup of $F(G)$. Note that $L = O_\sigma(F(G))$, and so $L$ is normal in $G$. Observe that $G = F(G)H = LNH = LH$. Consequently, $G/L \cong H/(H \cap L)$.The subgroups $H/N$ and $F(G)/N$ intersect trivially, and so $H \cap F(G) = N$. By Dedekind's Lemma, \[N = H \cap F(G) = H \cap LN = (H \cap L)N.\] In particular, $(H \cap L) \le N$. Since $N$ is Hall subgroup of $F(G)$ by our hypothesis and $|L|$ is a $\sigma$-number, we know that $N \cap L = 1$ by order considerations. Now, $(H \cap L) \le (N \cap L) = 1$, and so $G/L \cong H$.


We proceed with the hypothesis that $H$ is not nilpotent. So, the hypotheses in (2) are therefore in effect. As $N$ is a normal Hall subgroup of $H'N$, we know that $N \cap H'$ is a Hall subgroup of the nilpotent group $H'$. Let $L$ be the complement to $N \cap H'$ in $H'$, and note that, in fact, $L = O_\pi(H')$, where $\pi$ is the set of prime divisors of $| H' : N \cap H'| = |H'N: N|$. So, $L$ is normal in $H$ as $L$ is characteristic in the subgroup $H'$, which in turn is normal in $H$. (In fact, $L$ is a characteristic subgroup of $H$.) Now, \[H = H'NU = H'U = L(N \cap H')U = LU.\] Since $U/N$ is a complement to $H'N/N$, we know that the subgroups $U/N$ and $H'N/N$ intersect trivially. Thus, $U \cap H'N = N$. Note that $H'N = L(N \cap H')N = LN$, and so by Dedekind's Lemma \[ N = U \cap LN = (U \cap L)N.\] Hence, $(U \cap L) \le N$, and so $(U \cap L) \le (N \cap L) = 1$, where we used the fact that $|L|$ and $|N|$ are coprime. Thus, $H/L \cong U$. By (2), $U$ is nilpotent; hence, $H/L$ is nilpotent. Now, bearing in mind that $L$ is a $p'$-subgroup of $H$, we have that the subgroup $V$ is a normal $p$-complement of $H$, where $V/L = O_{p'}(H/L)$.
\end{proof}

In the hypotheses of Theorem \ref{Lemma 1.26.}, we have a nilpotent normal subgroup $N$ of a group $G$ such that $G/N$ is a Frobenius group with Frobenius kernel $F(G)/N$.  Since the Frobenius kernel of a Frobenius group is the Fitting subgroup of the group, we actually have that $$F(G/N) = F(G)/N.$$  Nilpotent normal subgroups $N$ of a group $G$ that satisfy the equation $F(G/N) = F(G)/N$ were called {\it generalized Frattini subgroups} by Beidleman and Seo in \cite{beidlemanseo} and further studied by Beidleman in \cite{biedleman2}.  Beidleman and Dykes considered a further generalization of these subgroups in \cite{beidlemandykes}.

It is well known that ${\rm SL}_2 (3)$ has a unique non split extension by $Z_2$.  This group is isoclinic to ${\rm GL}_2 (3)$ and so, we denote it by $\widetilde{{\rm GL}_2 (3)}$.  This group occurs as a Frobenius complement and it has a Sylow $2$-subgroup that is generalized quaternion of order $16$.  We determine the tidiness of certain extensions of $S_4$.

\begin{theorem} \label{Lemma 1.27.} 
Suppose $G$ is a $\{ 2, 3 \}$-group that satisfies one of the following:
\begin{enumerate}
\item $O_2(G)$ is a Klein $4$-group, $G/O_3(G) \cong S_4$ and $G/O_2 (G)$ is a Frobenius group whose Frobenius kernel is the Sylow $3$-subgroup of $G/O_2(G)$ and whose Frobenius complement has order $2$.
\item $O_2 (G)$ is the Sylow $2$-subgroup of $G$ and is the quaternion group of order $8$ and $G/O_3(G) \cong {\rm SL}_2 (3)$.
\item $O_2 (G)$ is the quaternion group of order $8$, $G/O_3(G) \cong \widetilde {{\rm GL}_2 (3)}$ and $G/O_2 (G)$ is a Frobenius group whose Frobenius kernel is the Sylow $3$-subgroup of $G/O_2 (G)$ and whose Frobenius complement has order $2$.
\end{enumerate}
Then $G$ is a tidy group.
\end{theorem}

\begin{proof} 
By Lemma \ref{Lemma 1.22.}, it suffices to prove that ${\rm Cyc}_G (x)$ is a subgroup for $1 \ne x \in G$ when $x$ has either $2$-power or $3$-power order. Suppose first that $1 \ne x$ has $3$-power order, and let $P$ be a Sylow $3$-subgroup of $G$.  We claim that $C_G (x) \leq PO_2(G)$. In (2), we have $G = PO_2(G)$, so the claim holds trivially. In (1) and (3), we know $G/O_2(G)$ is a Frobenius group whose Frobenius kernel is $PO_2(G)/O_2(G)$. This implies that $C_{G/O_2(G)} (xO_2(G)) \le PO_2(G)/O_2(G)$, and so, $C_G (x) \le PO_2(G)$.  It follows that the claims holds in these cases also.  Thus, $C_G (x)$ has a normal $2$-complement.  Since either $P$ is cyclic or has exponent $3$, this implies ${\rm Cyc}_G (x)$ is a subgroup by either Lemma \ref{tidy exp p 1} or \ref{Lemma 1.19.}.

Now, suppose $1 \ne x$ has $2$-power order. If $o(x) = 4$, then $\langle x \rangle$ will be a Sylow $2$-subgroup of $C_G (x)$, and so, ${\rm Cyc}_G (x)$ is a subgroup by Lemma \ref{Lemma 1.19.}.  We suppose $o(x) = 2$. If a Sylow $2$-subgroup of G is generalized quaternion, then we have that ${\rm Cyc}_G (x)$ is a subgroup by Lemma \ref{Lemma 1.19.} again. We are left with the case where a Sylow $2$-subgroup of G is dihedral of order $8$. This implies that we are in hypothesis (1). It is not  difficult to see that $C_G (x)$ will contain a Sylow $2$-subgroup $T$ of $G$. Since $G/O_3 (G) \cong S_4$, we see that $O_3 (G)$ is the Sylow $3$-subgroup of $C_G (x)$.  In particular, $C_G (x)$ has a normal $2$-complement. This implies that $C_G (x)$ is a subgroup by Lemma \ref{Lemma 1.20.}, and the result is proved.
\end{proof}

Recall that a group $G$ is a $2$-Frobenius group if there exists normal subgroups $1 < K < N < G$ so that $N$ is a Frobenius group with Frobenius kernel $K$ and $G/K$ is a Frobenius group with Frobenius kernel $N/K$.  The group $S_4$ is an example of a $2$-Frobenius group.  One consequence of this next lemma is that $S_4$ is the only $2$-Frobenius group that is a tidy group.

\begin{lemma}\label{Lemma 1.13.}
Suppose $G$ is a tidy group, and suppose there exists a normal subgroup $L$ so that $G/L$ is a 2-Frobenius group.  If there exists a normal subgroup $M$ of $G$ so that $L \cap M = 1$ and $ML/L$ is the Fitting subgroup of $G/L$, then $M$ is a Klein $4$-subgroup, $G/L$ is isomorphic to $S_4$, and $|L|$ is odd.
\end{lemma} 

\begin{proof} We have $1 \le L < K < N < G$ so that $N/L$ is a Frobenius group with Frobenius kernel $K/L$ and $G/K$ is a Frobenius group with Frobenius kernel $N/K$. Observe that $K/L$ is the Fitting subgroup of $G/L$, so $K = ML$.  Since $M \cap L = 1$, we have $K = M \times L$.  Let $H/L$ be a Frobenius complement of $N/L$. By the Frattini argument, we have that $G/L = N/L N_{G/L}(H/L) = KH/L N_{G/L}(H/L) = K/L N_{G/L} (H/L)$. Observe that $K/L \cap N_{G/L} (H/L) = 1$, so $N_{G/L}(H/L) \cong G/K$ is a Frobenius group with Frobenius kernel $H/L$. Let $B/L$ be a Frobenius complement of $N_{G/L} (H/L)$. 
	
Let $p$ be a prime divisor of $|M|$, and let $P$ be a Sylow $p$-subgroup of $M$. Observe that $PHB/L$ is a 2-Frobenius group. Let $Q/L$ be a subgroup of $B/L$ of prime order $q$. Observe that $P HQ/L$ is a 2-Frobenius group.  Suppose $p \ne q$.  Observe that $HQ/L$ acts nontrivially on $P$ and $Q/L$ cannot be contained in the kernel of this action.  We have that $Q/L$ acts nontrivially on $P$.  On the other hand, $HQ/L$ is not a Frobenius complement, so we have that $C_P (Q/L) > 1$. We now see that $C_{Q/L} (P) = 1$ and $Q/L$ does not act Frobeniusly on $P$. This contradicts Lemma \ref{Lemma 1.6.}. Thus, we must have $p = q$, which implies that $M$ and $B/L$ are both $p$-groups.  Since $MH/L$ is a Frobenius group, we see that $M$ is not dihedral or generalized quaternion.  Since $HB/L$ is nonabelian and is isomorphic to a subgroup of $G/C_G (M) \le {\rm Aut} (M)$, it follows that $M$ must be cyclic. 

By Theorem \ref{p-groups}, $M$ must have exponent $p$.  This implies that $Z := Z(M)$ is elementary abelian, and we can view $Z$ as a module for $Z_p$. Observe also that $ZHQ/L$ is a $2$-Frobenius group. Now, we now apply Theorem 15.16 of \cite{text} to see that $Z$ has a basis which is permuted by $Q/L$ in orbits of size $p$. Let $\{a_1, \dots a_p \}$ be one of these orbits.  We see that $\langle a_1, Q/L \rangle$ is isomorphic to $Z_p \wr Z_p$. It is well-known that $Z_p \wr Z_p$ has exponent $p^2$ (see Problem 4A.7 in \cite{Isagrp}).  Since $M$ is not cyclic, we can use Theorem \ref{p-groups} to see that a Sylow $p$-subgroup of $G$ must have exponent $p$ when $p$ is odd. Hence, we must have that $p = 2$. Now, the Sylow $2$-subgroup of G is not abelian, and has a normal subgroup that is elementary abelian.  In light of Theorem \ref{p-groups} again, we deduce that a Sylow $2$-subgroup of $G$ is dihedral of order $8$ and $M = Z_2 \times Z_2$.  Since ${\rm Aut} (M) \cong {\rm GL}_2 (2)$, we have $H/L \cong Z_3$ and $HB/L \cong S_3$. We finally conclude that $G/L \cong S_4$ and $|L|$ is odd.
\end{proof}

\section{Tidy $\{ p, q\}$-groups} \label{secntidypq}

In this section, we classify the tidy $\{p , q \}$-groups.  We first consider the case where the Sylow $2$-subgroup is generalized quaternion and $O_2 (G)$ has order $2$.  

\begin{lemma}\label{Lemma 1.10.} 
Let $G$ be a solvable tidy group. Suppose a Sylow $2$-subgroup $T$ of $G$ is generalized quaternion. If $|O_2 (G)| = 2$, then $T$ is a quaternion group of order $8$, $G$ has a normal $2$-complement, and $|F(G) : C_{F (G)}(T)|$ is divisible by at least two primes.	
\end{lemma}

\begin{proof} 
The hypothesis that $|O_2 (G)| = 2$ implies that $Z = O_2(G) \le Z(G)$. Since $G$ is not abelian, we see that $Z(G) \ne F(G)$. Thus, $O_{2'} (G) > 1$. In particular, there is an odd prime $p$ so that $O_p (G) > 1$. Let $p_1, \dots, p_n$ be the odd prime divisors of $|F(G)|$ so that $F(G) = Z \times F^{2'} (G)$ where $F^{2'} (G) = O_{p_1} (G) \times \cdots \times O_{p_n} (G)$ is the Hall $2$-complement of $F(G)$.  Since $Z$ is central, we know that $C_G (F^{2'} (G)) = C_G (F(G)) \leq F(G)$, so $C_T (F^{2'} (G)) \leq T \cap F(G) = Z$. Since $Z \leq C_T (F^{2'} (G))$, we deduce that $C_T (F^{2'} (G)) = Z$. Observe that no quotient of $T/Z$ is generalized quaternion.  

By Lemma \ref{Lemma 1.6.}, we see that every element of $T/C_T(O_{p_i}(G))$ acts Frobeniusly on $O_{p_i} (G)$, and so, $T/C_T(O_{p_i} (G))$ is a Frobenius complement. Since it is not generalized quaternion, we conclude that $T/C_T(O_{p_i} (G))$ is cyclic for each value of $i$.  This implies that $T' \le C_T (O_{p_i}(G))$ and thus, $T' \le \cap_{i=1}^n C_T(O_{p_i} (G)) = C_T(F^{2'} (G)) = Z$. This implies that $T$ has nilpotence class $2$.  The only generalized quaternion group with nilpotence class $2$ is the quaternion group of order $8$, and thus, $T$ is the quaternion group. Notice that $T' < C_T (O_{p_i} (G))$, so we would have a contradiction if only one prime divides $|F(G) : C_{F (G)}(T)|$. Thus, at least two primes divide $|F(G) : C_{F (G)}(T)|$.

To complete the proof, we need to show that $G$ has a normal $2$-complement.  Let $H$ be a $2$-complement of $G$. If $C_T (O_{p_i} (G)) < T$ for some $i$, then $G/C_G (O_{p_i} (G))$ has a Sylow $2$-subgroup of order $2$.  By a corollary to the Burnside normal $p$-complement theorem (Corollary 5.14 of \cite{Isagrp}), this implies that $G/C_G(O_{p_i} (G))$ has a normal $2$-complement.  Hence, we have $HC_G (O_{p_i} (G))$ is normal in $G$ for each such $i$.  On the other hand, if $C_T (O_{p_i} (G)) = T$, then $HC_G (O_{p_i} (G)) = G$.  We deduce that $HC_G (O_{p_i} (G))$ is normal in $G$ for all $i$.  This implies that $\cap_{i=1}^n HC_G(O_{p_i} (G)) = H(\cap_{i=1}^n C_G(O_{p_i}(G)))$ is a normal subgroup of $G$. Observe that $HZ \le H \cap_{i=1}^n C_G(O_{p_i}(G)) = H C_G(F^{2'} (G)) \le HF(G) = HZ$. Thus, $HZ$ is normal in $G$. Since $Z$ is central, we see that $H$ is normal and thus characteristic in $HZ$. We now obtain the desired conclusion that $H$ is normal in $G$.
\end{proof} 

We next consider the centralizer $C_H (O_2(G))$ for a group $G$ where a Sylow $2$-subgroup of $G$ is generalized quaternion group and $H$ is a $2$-complement of $G$.  Note the similarity to Theorem \ref{Lemma 1.22.}.

\begin{lemma} \label{Lemma 1.12.} 
Let $G$ be a tidy group. Suppose that a Sylow $2$-subgroup of $G$ is a generalized quaternion group. Let $H$ be a $2$-complement of $G$ and let $C = C_H (O_2(G))$. Then $C$ is normal in $G$.
\end{lemma}

\begin{proof} 
Since $C_G (O_2(G))$ is normal in $G$, we see that $C$ is normal in $H$.  It follows that $C$ is a Hall $2$-complement of $C_G (O_2 (G))$, and $G/C_G (O_2(G))$ is isomorphic to a subgroup of ${\rm Aut} (O_2 (G))$.  We know that $O_2 (G)$ is either generalized quaternion or cyclic. If $O_2 (G)$ is not the quaternion group, then ${\rm Aut} (O_2(G))$ is a $2$-group.  Let $T$ be a Sylow $2$-subgroup of $G$, and observe that $T \cap C_G (O_2(G))$ is a Sylow $2$-subgroup of $C_G (O_2 (G))$. In particular, $O_2 (G)$ is a normal subgroup of $T$. Now, $T$ has a cyclic subgroup $D$ of index $2$.  If $O_2 (G)$ is generalized quaternion, then $O_2 (G)$ will contain elements outside of $D$, and it will intersect $D$ in a subgroup of order at least $4$.  Observe that if $x \in T \setminus D$, then $C_T (x) = \langle x \rangle$, and if $x \in D$, then $C_T (x) = D$. We deduce that if $O_2 (G)$ is generalized quaternion, then $C_T (O_2 (G)) = Z(T)$. 

Since $O_2 (G)$ is normal in $T$, if $O_2 (G)$ is cyclic, then either $O_2 (G) \leq D$ or $C_T (O_2(G)) = O_2 (G)$. Hence, when $O_2(G)$ is cyclic of order at least $4$, either $C_T (O_2 (G)) = D$ or $C_T (O_2 (G)) = O_2(G)$. Thus, when $O_2 (G)$ is generalized quaternion or cyclic of order at least $4$, the Sylow $2$-subgroup of $C_G (O_2 (G))$ is cyclic. By Lemma \ref{tidy cyclic 1}, $C_G (O_2 (G))$ has a normal $2$-complement. This implies that $C$ is normal in $C_G (O_2 (G))$ and hence characteristic in $C_G (O_2 (G))$. Since $C_G (O_2 (G))$ is normal in $G$, we deduce that $C$ is a normal subgroup of $G$.  The remaining possibility is that $O_2 (G)$ has order $2$. This implies that $C = H$. By Lemma \ref{Lemma 1.10.}, we conclude that $H$ is normal in $G$, and the result is proved.
\end{proof}

In this next lemma, we classify the groups where a Sylow $2$-subgroup is generalized quaternion and there is a central element of order $2$.  This is a key step in the general classification.

\begin{lemma}\label{Lemma 1.11.}
Let $G$ be a solvable, tidy group. Suppose a Sylow $2$-subgroup of $G$ is generalized quaternion. If $x \in Z(G)$ has order $2$, then one of the following occurs:
\begin{enumerate}
\item $G$ has a normal $2$-complement.
\item $O_2 (G)$ is the quaternion group of order $8$ and if $H$ is a Hall $\{ 2, 3\}$-subgroup of $G$, then $H/O_3 (H)$ is isomorphic to ${\rm SL}_2 (3)$ or $\widetilde{{\rm GL}_2 (3)}$ and there is a Hall $\{ 2, 3\}$-complement $T$ so that $O_3(H) T$ is a normal subgroup of $G$.
\end{enumerate}	
\end{lemma} 

\begin{proof} 
Since $G$ has a central element of order $2$, we must have that $O_2 (G) > 1$. If $|O_2 (G)| = 2$, then by Lemma \ref{Lemma 1.10.}, we see that $G$ has a normal $2$-complement. Suppose that $O_2 (G)$ is cyclic or is generalized quaternion of order at least $16$.  We know that $C_G (O_2(G))$ is normal in $G$ and $G/C_G (O_2 (G))$ is isomorphic to a subgroup of ${\rm Aut} (O_2 (G))$.  This implies that $G/C_G (O_2 (G))$ is a $2$-group.  By Lemmas \ref{tidy cyclic 1} or \ref{Lemma 1.12.}, $C_G (O_2 (G))$ has a
normal $2$-complement $C$.  Observe that $C$ will be characteristic in $G_G (O_2 (G))$ and hence it is normal in $G$.  Since $|G:C_G (O_2 (G))|$ is a $2$-power $C$ is a $2$-complement of $G$.
  
Thus, we are left with the case that $O_2 (G)$ is the quaternion group of order $8$. We know that $G/C_G (O_2 (G))$ is isomorphic to a subgroup of ${\rm Aut} (O_2 (G)) \cong S_4$. Notice that the Sylow $2$-subgroup of $C_G (O_2 (G))$ will be $Z (O_2 (G))$ which has order $2$, and $$\frac {O_2 (G)C_G (O_2 (G))}{C_G (O_2 (G))} \cong Z_2 \times Z_2.$$ Thus, there are three possibilities here. In particular, $G/C_G (O_2(G))$ is isomorphic to (1) a Klein $4$-group, (2) $A_4$, or (3) $S_4$.  We may use Lemma \ref{tidy cyclic 1} to see that $C_G (O_2 (G))$ has a normal $2$-complement $C$.  In the first case, it is not difficult to see that $C$ is a normal $2$-complement of $G$. 

We assume that we are in cases (2) or (3). Let $H$ be a Hall $\{ 2, 3 \}$-subgroup of $G$ and take $T$ to be a Hall $\{ 2,3 \}$-complement of $C$. Observe that $H \cap C$ will be a $3$-subgroup, will be normal in $H$ and will be a Sylow $3$-subgroup of $C$. Also, $C_H (O_2 (G)) = Z (O_2 (G)) \times (H \cap C)$. We know that $H/C_H (O_2 (G)) \cong G/C_G (O_2 (G))$ is either $A_4$ or $S_4$. In either case, we see that $H \cap C = O_3 (H)$ and $C = O_3(H) T$. Since a Sylow $2$-subgroup of $H$ is either the quaternion
group of order $8$ or the generalized quaternion group of order $16$, we see that $H/O_3 (H)$ is isomorphic to either ${\rm SL}_2 (3)$ or $\widetilde{{\rm GL}_2 (3)}$. Therefore, we obtain conclusion (2).
\end{proof} 

We also need to understand the centralizer when a Sylow subgroup is cyclic.  We obtain additional information in this case.
We define the terms of the upper central series as follows.  Set $Z_0 (G) = 1$ and for $i \ge 1$, define $Z_i (G)/Z_{i-1} (G) = Z (G/Z_{i-1} (G))$.  It is easy to see that $Z_1 (G) = Z(G)$ and $Z_0 (G) \le Z_1 (G) \le \dots$.  Since $G$ is finite, there is an integer $N$ so that $Z_n (G) = Z_N (G)$ for all $n \ge N$.  We say that $Z_N (G)$ is the {\it hyper-center} of $G$, and we denote it by $Z_\infty (G)$. 

\begin{lemma} \label{Lemma 1.5}
Let $G$ be a tidy group and let $p$ be a prime.   Suppose that $G$ has a Sylow $p$-subgroup $P$ that is cyclic and let $H$ be a Hall $p$-complement of $G$.  If $O_p (G) > 1$ and $C = C_H (O_p(G))$, then $C$ is normal in $G$. In addition, if $C < H$, then $G/C$ is a Frobenius group with Frobenius kernel $P \cong PC/C$ and Frobenius complement $H/C$ that is cyclic of order dividing $p-1$, $P \cap Z_{\infty}(G) = 1$.
\end{lemma}

\begin{proof}
Notice that $C_G (O_p (G))$ is normal in $G$, so $H \cap C_G (O_p (G)) = C$ is a Hall $p$-complement of $C_G (O_p (G))$.  
There is an element $x$ so that $O_p (G) = \langle x \rangle$.  Then $C_G (x) = C_G (O_p (G))$.  By Lemma \ref{tidy cyclic 1}, $C_G (x)$ has a normal $p$-complement.  This implies that $C$ is normal and hence, characteristic in $C_G (O_p (G))$.  Since $C_G (O_p (G))$ is normal in $G$, we conclude that $C$ is normal in $G$.
	
	
Suppose $C < H$.  Observe that $P C = C_G (O_p (G))$. Also, $H/C \cong G/C_G (O_p (G))$ is isomorphic to a subgroup of ${\rm Aut} (O_p (G))$.  Since $O_p (G)$ is cyclic of prime power order and $|H:C|$ is coprime to $p$, we conclude that $H/C$ is cyclic of order dividing $p-1$.  (Recall that the automorphism group of a cyclic $2$-group is a $2$-group and when $p$ is odd, the automorphism group of a cyclic group of order $p^a$ is cyclic of order $\phi (p^a) = p^{a-1}(p-1)$.)  It is not difficult to see that any nontrivial automorphism of coprime order of a cyclic group of prime power order will be fixed point-free.  (This strongly uses the fact that we have prime power order; it is not true if you do not assume this.)  Thus, $O_p(G)H/C$ is a Frobenius group with Frobenius kernel $O_p (G)C/C$ and Frobenius complement $H/C$.    This implies $Z(G) \cap O_p (G) = 1$.  Notice that $Z (G) \cap P$ is a normal $p$-subgroup of $G$, and so, we deduce that $Z(G) \cap P \le O_p (G)$ and so, $Z(G) \cap P = 1$.  This implies that $Z_{\infty} (G) \cap P = 1$.
	
Note that $|G:PC| = |H:C| = |H:O_p (G)C|$ divides $p-1$.  Since the number of Sylow $p$-subgroups of $G/C$ needs to be congruent to $1$ modulo $p$, we deduce that $PC/C$ is normal in $G/C$.  Hence, we have that $PC$ is normal in $G$.  By Fitting's Lemma, we see that $H/C$ acts Frobeniusly on $PC/C$.
\end{proof}

We also make use of the following technical lemma about hypercenters.

\begin{lemma} \label{center}
Let $G$ be a tidy solvable group and let $p$ be a prime.  Suppose that $G$ has a Sylow $p$-subgroup $P$ and let $H$ be a Hall $p$-complement of $G$.  If $O_p (G) > 1$ and $H \le C_G (O_p (G))$, then $Z (G) \cap P = Z (P) \cap (O_p (G)) > 1$.  In particular, $O_p (G) = Z_{\infty} (G) \cap P$.
\end{lemma}

\begin{proof}
Observe that $G = PH$.  Since $O_p (G)$ is a nontrivial normal subgroup of $P$, we have that $Z(P) \cap O_p (G) > 1$. Observe that both $P$ and $H$ centralize $Z(P) \cap O_p (G)$, so $Z(P) \cap O_p (G) \le Z(G) \cap P$.  On the other hand, $Z(G) \cap P \le P$, so $Z(G) \cap P \le Z(P)$ and $Z(G) \cap P$ is a normal $p$-subgroup of $G$, so $Z(G) \cap P \le O_p (G)$.  We obtain $Z(G) \cap P \le Z(P) \cap O_p (G)$.  We deduce that $Z(G) \cap P = Z(P) \cap O_p (G)$. 
	
We claim that $Z_n (G) \cap P = Z_n (P) \cap O_p (G)$ for $n \ge 1$.  We have proved this for $n = 1$.  Suppose we know it is true for $m \ge 1$, and we now prove it is true for $m + 1$.  Since $H$ centralizes $O_p (G)$, it follows that $HZ_m (G)$ will centralize $(Z_{m+1} (P) \cap O_p (G))Z_m (G)/Z_m (G)$.  It is obvious that $PZ_m (G)/Z_m (G)$ centralizes $(Z_{m+1} (P) \cap O_p (G))Z_m (G)/Z_m (G)$.  We have that $Z_{m+1} (P) \cap O_p (G) \le Z_{m+1} (G) \cap P$.  On the other hand, it is easy to see that $Z_{m+1} (G) \cap P$ is contained in $Z_{m+1} (P)$ and it is a $p$-group that is normal in $G$, so we have proved the claim.  Since $O_p (G) \le Z_n (P)$ for some $n$, we have $O_p (G) \le Z_n (G) \le Z_{\infty} (G)$ as desired.
\end{proof}

We now arrive at the main theorem of this section, the classification of tidy $\{ p, q \}$-groups.

\begin{theorem} \label{Lemma 1.17.}
Suppose $G$ is a $\{ p, q \}$-group for distinct primes $p$ and $q$. Then $G$ is tidy if and only if $G$ has tidy Sylow $p$- and Sylow $q$-subgroups and one of the following occurs:
\begin{enumerate}
\item $G$ is nilpotent.
\item Up to relabeling $p$ and $q$, $Z_\infty (G)$ is a $q$-group and $G/Z_\infty (G)$ is a Frobenius group whose Frobenius kernel is the Sylow $p$-subgroup.
\item $\{ p, q \} = \{ 2, 3\}$, $O_2 (G)$ is a Klein $4$-group, $G/O_3(G) \cong S_4$ and $G/O_2 (G)$ is a Frobenius group whose Frobenius kernel is the Sylow $3$-subgroup of $G/O_2 (G)$ and whose Frobenius complement has order $2$. Also, $Z(G) = 1$.
\item $\{ p, q \} = \{ 2, 3\}$, $O_2 (G)$ is a Sylow $2$-subgroup of $G$ and is the quaternion group of order $8$, $G/O_3 (G) \cong {\rm SL}_2 (3)$. Also, $Z_\infty (G) = Z(O_2 (G)) \times O_3(G)$.
\item $\{ p, q \} = \{ 2, 3\}$, $O_2 (G)$ is the quaternion group of order $8$, $G/O_3(G) \cong \widetilde{{\rm GL}_2 (3)}$ and $G/O_2 (G)$ is a Frobenius group whose Frobenius kernel is the Sylow $3$-subgroup of $G/O_2 (G)$ and whose Frobenius complement has order $2$. Additionally, $Z_\infty (G) = Z(G) = Z(O_2(G))$.
\end{enumerate}
\end{theorem}

\begin{proof} 
We begin by assuming that $G$ is a tidy group. This implies that $G$ has tidy Sylow $p$- and Sylow $q$-subgroups $P$ and $Q$ respectively. If both $P = O_p (G)$ and $Q = O_q (G)$, then $G = P × Q$ and $G$ is nilpotent and we have (1). We now assume that $G$ is not nilpotent; so both $P > 1$ and $Q > 1$ and either $O_p (G) < P$ or $O_q (G) < Q$. Without loss of generality, we assume $O_q (G) < Q$. Notice that if $O_p(G) = 1$, then $O_p(G) < P$. Thus, we may assume that $O_p (G) > 1$ and $O_q (G) < Q$.  Furthermore, if both $O_p (G) > 1$ and $1 < O_q (G) < G$ and $2 \in \{ p, q \}$, we set $p = 2$.  

Set $C = C_Q (O_p (G))$.  By Theorem \ref{p-groups}, we know that $P$ has exponent $p$, is cyclic, is dihedral, or is generalized quaternion.  By Lemmas \ref{tidy exp p 2}, \ref{Lemma 1.8.}, \ref{Lemma 1.12.}, and \ref{Lemma 1.5} 
$C$ is normal in $G$.  This implies that $C = O_q (G)$.  By our assumption, $C < Q$.  In particular, $G$ does not have a normal $p$-complement.  Hence, if $P$ is a dihedral group, then $p = 2$.  Using Lemma \ref{Lemma 1.7.},
we have $O_2 (G) = Z_2 \times Z_2$ and in view of Lemma \ref{Lemma 1.8.},
$G/C \cong S_4$.  This implies that $q = 3$ and $C = O_3 (G)$.  Observe that $C_P (O_3 (G)) = O_2 (G)$.  Notice that $Q O_2 (G)$ is a normal subgroup of $G$.  Obviously, $P/O_2 (G)$ acts Frobeniusly on $Q/O_3 (G)$, since $G/O_2 (G) O_3 (G) \cong S_3$.  By Lemma \ref{Lemma 1.6.}, we have that $P/O_2 (G)$ acts Frobeniusly on $O_3 (G)$.  It follows that $G/O_2 (G)$ is a Frobenius group with Frobenius kernel $QO_2 (G)/O_2 (G)$ and a Frobenius complement of order $2$.  It is not difficult to see that $Z(G) = 1$.  This yields (3).

Next, suppose $P$ is generalized quaternion, so again $p = 2$.  We are assuming that $O_2 (G) > 1$.  This implies that $O_2 (G)$ is either cyclic or generalized quaternion.  In either case, $O_2 (G)$ has a characteristic subgroup of order $2$.  This implies that there exists $x \in Z(G)$ with $o (x) = 2$.  Since we are assuming that $G$ does not have a normal $2$-complement, we may apply Lemma \ref{Lemma 1.11.} to see that $O_2 (G)$ is the quaternion group of order $8$ and $q = 3$ and $G/O_3 (G)$ is isomorphic to either ${\rm SL}_2 (3)$ or $\widetilde{{\rm GL}_2 (3)}$.  Now, if $G/O_3 (G)$ is isomorphic to ${\rm SL}_2 (3)$, then $O_2 (G)$ will be a Sylow $2$-subgroup of $G$.  This implies that $P = C_P (O_3 (G))$.  By Lemma \ref{center}, we see that $Z_\infty (G)\cap Q = O_3 (G)$.  Because $G/(Z(O_2 (G)) O_3 (G)) \cong A_4$, we deduce that $Z_\infty (G) \cap O_2 (G) = Z(O_2 (G))$.  This proves (4).  

Suppose $G/O_3 (G)$ is isomorphic to $\widetilde{{\rm GL}_2 (3)}$.  Observe that $O_2 (G) = C_P (O_3 (G))$. Notice that $Q O_2 (G)$ is a normal subgroup of $G$.  Obviously, $P/O_2 (G)$ acts Frobeniusly on $Q/O_3 (G)$, since $G/O_2 (G) O_3 (G) \cong S_3$.  By Lemma \ref{Lemma 1.6.}, we have that $P/O_2 (G)$ acts Frobeniusly on $O_3 (G)$.  It follows that $G/O_2 (G)$ is a Frobenius group with Frobenius kernel $QO_2 (G)/O_2 (G)$ and a Frobenius complement of order $2$.  This implies that $Z_\infty (G) \cap Q = 1$.  Since $G/(Z(O_2 (G)) O_3 (G)) \cong S_4$, we deduce that $Z_\infty (G) \cap P = Z(O_2 (G))$.  This proves (5).  

We now assume $P$ is cyclic or has exponent $p$.  By Lemma \ref{Lemma 1.5} or \ref{tidy exp p 2}, 
we see that $QO_p (G)/O_q (G)$ is a Frobenius group with Frobenius kernel $PO_q (G)/O_q (G)$.  Also, $P \cap Z_\infty (G) = 1$.  Observe that $P = C_P (O_q (G))$.  By Lemma \ref{center}, we have $O_q (G) = Z_\infty (G) \cap Q$.  We conclude that $Z_\infty (G) = O_q (G)$.   If $G = O_p (G) Q$, then we obtain (2).   Thus, we may assume that $O_p (G) Q < G$, and this implies that $O_p (G) < P$.  Let $N$ be a normal subgroup so that $O_p (G)O_q (G) < N$ and $N/O_p (G) O_q (G)$ is a chief factor for $G$.  Thus, $O_p (G) O_q (G)$ is either a $p$-group or a $q$-group.  

Suppose first that $N/(O_p (G)O_q (G))$ is a $q$-group.  (Note that we may have $O_q (G) = 1$.)  Since $QO_p (G)/O_q (G)$ is a Frobenius group, we have that $N/O_q (G)$ is a Frobenius group with $O_p (G) O_q (G)$ as its Frobenius kernel.  By the Frattini argument, we have $G = N N_G (N \cap Q) = O_p (G) (N \cap Q) N_G (N \cap Q)  = O_p (G) N_G (N \cap Q)$.  It follows that $P = O_p (G) N_P (N \cap Q)$.  Since $(N \cap Q)/O_q (G)$ acts Frobeniusly on $O_p (G)$, we see that $O_p (G) \cap N_P (N \cap Q) = 1$.  Hence, we can find $1 \ne x \in N_P (N \cap Q)$.  Suppose there exists $y \in N \cap Q$ that is centralized by $x$.  Observe that $1 < Z (P) \cap O_p (G) \le C_G (x)$.  By Lemma \ref{tidy exp p 1} or \ref{tidy cyclic 1}, we know that $C_G (x)$ has a normal $p$-complement.  This implies that $y$ is in a normal subgroup of $C_G (x)$ that is disjoint from $Z(P) \cap O_p (G)$.  It follows that $y$ centralizes $Z(P) \cap O_p (G)$.  Since $(N \cap Q)/O_ q(G)$ is acting Frobeniusly on $O_p (G)$, we deduce that $y \in O_q (G)$.  We now see that $N_P (N \cap Q)$ is acting Frobeniusly on $(N \cap Q)/O_q (G)$.  Since $N/(O_p (G) O_q (G)) \cong (N \cap Q)/O_q (G)$, we see that $NP/(O_p (G) O_q (G))$ is a Frobenius group with Frobenius kernel $N/(O_p (G) O_q (G))$.  Hence, $NP/O_q (G)$ is a $2$-Frobenius group.  Observe that $F(NP) = O_p (G) \times O_q (G)$.  Thus, Lemma \ref{Lemma 1.13.} applies and implies that $p = 2$ and $NP/O_p (G) \cong S_4$.  This however contradicts the fact that a Sylow $p$-subgroup is cyclic or has exponent $p$.  This case cannot occur.

We have that $N/(O_p (G) O_q (G))$ is a $p$-group.  This implies that $O_q (G) > 1$ since $N > O_p (G) O_q (G)$.  Notice that $O_p (G) < P$ and $1 < O_q (G) < Q$ implies that if $2 \in \{ p, q \}$, then $p = 2$.  Observe that $N \cap P$ is not normal in $G$, so $N \cap P$ cannot centralize $O_q (G)$.  Using Lemma \ref{Lemma 1.6.}, we see that $(N \cap P)/O_p (G)$ acts Frobeniusly on $O_q (G)$.  (Since $q \ne 2$, we have that conclusion (3) of Lemma \ref{Lemma 1.6.} does not occur.)  Since $q \ne 2$, we know from Theorem \ref{p-groups} that $Q$ is either cyclic or exponent $q$.  If $Q$ is cyclic, then Lemma \ref{Lemma 1.5} would apply and $O_q (G) = Q$ which is a contradiction to the choice of $p$ and $q$.  Hence $Q$ has exponent $q$.  Now, we can reverse the roles of $p$ and $q$, and the same argument as in the last paragraph applies to yield a contradiction.  This completes this direction.

To see the converse, observe that a nilpotent group where all the Sylow subgroups are tidy is tidy, so we have the result if (1) holds.  If (3), (4), or (5) hold, then $G$ is tidy by Theorem \ref{Lemma 1.27.}.  If (2) holds, then we have that $G$ is tidy by Theorem \ref{Lemma 1.26.} where $N = Z_\infty (G)$ and $H = Q$ is nilpotent. 
\end{proof}

\section{Solvable tidy groups}

In this section, we work to prove some consequences of our characterization of solvable tidy groups.

We now work to prove that the quotients of finite solvable tidy groups are tidy.  We will see that the proof relies on the fact that a finite solvable group is tidy if and only if its Hall $\{p,q \}$-subgroups are tidy for all primes $p$ and $q$ and on the classification of tidy $\{ p, q \}$-groups.  At this time, it is an open question as to whether quotients of finite nonsolvable tidy groups are tidy.

We begin with the observation using Theorem \ref{p-groups} that any quotient of a tidy $p$-group will be tidy where $p$ is a prime.  Next, we use Theorem \ref{Lemma 1.17.} to determine the tidy $\{ p, q \}$-group where $p$ and $q$ are distinct primes.  This will be the key to the next lemma.

\begin{lemma}\label{pq-quotients}
Let $G$ be a tidy $\{p,q\}$-group for distinct primes $p$ and $q$.  If $N$ is a normal subgroup of $G$, then $G/N$ is a tidy group.
\end{lemma}

\begin{proof}
We know that all subgroups of tidy groups are tidy, so the Sylow $p$- and Sylow $q$-subgroups of $G$ are tidy.  Hence, if $G/N$ is a $p$-group or a $q$-group, then it is a quotient of a tidy $p$-group or $q$-group, and we have seen that quotients of tidy $p$-groups are tidy.  Thus, we may assume that $G/N$ is not a $p$-group or a $q$-group.  We now consider the possibilities from Theorem \ref{Lemma 1.17.}.  If $G$ is nilpotent, then $G/N$ is nilpotent, and the Sylow $p$- and Sylow $q$-subgroups are quotients of tidy groups, so they are tidy.  Thus, $G/N$ is tidy.  Suppose $Z_\infty (G)$ is a $q$-group and $G/Z_\infty (G)$ is a Frobenius group.  Let $F = F(G)$ and observe that $F = P \times Z_\infty (G)$ where $P$ is the Sylow subgroup.  Then $FN/N = PN/N \times Z_\infty (G) N/N$ and $G/(Z_\infty (G) N)$ will be a Frobenius group with Frobenius kernel $FN/N$.  By Theorem \ref{Lemma 1.26.}, we see that $G/N$ is tidy.  Finally, suppose $G$ satisfies (3), (4), or (5) of Theorem \ref{Lemma 1.17.}.  Observe that the possible quotients of $G$ that are not $p$-groups or $q$-groups are $S_3$, $A_4$, or a group that satisfies the hypotheses of Theorem \ref{Lemma 1.27.}.  Certainly, one can use Proposition 2.5 of \cite{ctidynote} to see that $S_3$ and $A_4$ are tidy, and Theorem \ref{Lemma 1.27.} shows that remaining possible quotients are tidy.
\end{proof}

We now prove that quotients of solvable tidy groups are tidy.  


\begin{proof}[Proof of Theorem \ref{quotient}]
	
Let $\pi$ be the set of primes that divide $|G|$, and let $\rho \subseteq \pi$ have size $2$.  Let $H$ be a Hall $\rho$-subgroup of $G$.  Then $HN/N$ is a Hall $\rho$-subgroup of $G/N$.  Since $G$ is tidy, we see that $H$ is tidy.  By Lemma \ref{pq-quotients}, we see that $HN/N$ is tidy.  Thus, for every two element subset $\rho$ of $\pi$, we see that $G/N$ has a tidy Hall $\rho$-subgroup.  By Theorem 1.1 of \cite{pre2},
we see that $G/N$ is tidy.
\end{proof}

We next work to prove that the Fitting height of a tidy solvable group is at most $4$.  We first prove a theorem that classifies the groups that arise modulo the centralizer of $O_p (G)$.  This is the key step in bounding the Fitting height.

\begin{theorem}\label{Lemma 1.14.}
Let $G$ be a solvable tidy group, let $p$ be a prime, let $H$ be a Hall $p$-complement of $G$ and let $C = C_H (O_p (G))$. If $O_p (G) > 1$, then $C$ is a normal subgroup of $G$ and one of the following occurs:
\begin{enumerate}
\item $G/C$ is a $p$-group.
\item $G/C$ is a Frobenius group whose Frobenius kernel is $O_p (G)C/C$.
\item $p = 2$, $O_2 (G)$ is the quaternion group of order $8$, and $G/C$ is isomorphic to ${\rm SL}_2 (3)$ or $\widetilde {{\rm GL}_2 (3)}$.
\item $p = 2$, $O_2 (G)$ is the Klein $4$-group and $G/C$ is isomorphic to $S_4$.
\item The following occur:
\begin{enumerate}
	\item $p = 3$.
	\item $G/C$ is a Frobenius group whose Frobenius kernel is a Sylow $3$-subgroup and whose Frobenius complement has order $2$.
	\item $O_2 (G)$ is either a Klein $4$-group or is a quaternion group of order $8$.
	\item $O_3 (G)$ has index $3$ in a Sylow $3$-subgroup of $G$.
	\item $G/O_3(G)C \cong S_3$.
\end{enumerate} 
\end{enumerate}
\end{theorem}

\begin{proof} 
We know by Theorem \ref{p-groups} that a Sylow $p$-subgroup of $G$ is one of the following: exponent $p$, cyclic, dihedral, or generalized quaternion.  Thus, applying Lemmas \ref{tidy exp p 2}, \ref{Lemma 1.8.}, 
\ref{Lemma 1.12.} and \ref{Lemma 1.5},
we see that $C$ is normal in $G$. 
%
Suppose first $p = 2$ and that a Sylow $2$-subgroup is a generalized quaternion group.  Since $O_2 (G) > 1$, this implies that $G$ has a central element of order $2$.  We may apply Lemma \ref{Lemma 1.11.} to see that either conclusion (1) or (3) occurs.
%
%
Next, suppose that $p= 2$ and $G$ has a dihedral Sylow $2$-subgroup.  By Lemma \ref{Lemma 1.8.}, 
we see that either conclusion (1) or (4) holds.  

Next, suppose that a Sylow $p$-subgroup is cyclic or has exponent $p$.  Let $q$ be a prime distinct from $p$ that divides $|G:C|$.  If either $p \ne 3$ or $q \ne 2$ when $p = 3$, we see from Theorem \ref{Lemma 1.17.} that $PQC/C$ is either nilpotent or a Frobenius group where $P$ is normalized by $Q$.  Notice that if $PQC/C$ is nilpotent, then $Q$ would centalize $O_p (G)$ and this would imply that $Q \le C$ contradicting $q$ divides $|G:C|$.   Hence, we have that $PQC/C$ is a Frobenius group.  Since all the Sylow subgroups of $G/C$ normalize $P$, we conclude that $P = O_p (G)$.  Notice that if $1 \ne x \in P$ centralizes a nontrivial element of $H/C$, then it will centralize a nontrivial element of prime power order, and this is a contradiction.  Hence, we have $G/C$ is a Frobenius group with Frobenius kernel $P = O_p (G) \cong O_p (G)C/C$.   This yields (2).
	
We are left with case that $p = 3$ and $2$ divides $|G:C|$.  Let $T$ be a Sylow $2$-subgoup of $G$.  If $PTC/C$ is nilpotent or a Frobenius group with Frobenius kernel $P$, then we can use the argument in the previous paragraph again to see that we have (2) again.  This handles conclusions (1) and (2) of Theorem \ref{Lemma 1.17.}.  We now suppose that $PTC/C$ satisfies (3), (4), or (5) from Theorem \ref{Lemma 1.17.}.  In all three of these cases, we see that $|P:O_3 (G)| = 3$.  We have seen that if $q$ is a prime divisor of $|G:C|$ other than $2$ or $3$, then $QC/C$ will act Frobeniusly on $P$.  This would imply that it acts Frobeniusly on $P/O_3 (G)$ which is a contradiction since $|P:O_3 (G)| = 3$.  Hence, no prime other than $2$ or $3$ divides $|G:C|$.  In conclusion (4) of Theorem \ref{Lemma 1.17.}, we see that $G/C$ is a $3$-group, and we have conclusion (1).  In conclusions (3) or (5) of Theorem \ref{Lemma 1.17.}, we obtain conclusion (5).  This proves the theorem.	
\end{proof}

We now are able to bound the Fitting height.  We also bound the derived length of $G/F(G)$.  Since one can find Frobenius kernels of exponent $p$ and arbitrarily large derived length, we are not able to bound the derived length of $G$.  In particular, we prove Theorem \ref{Lemma 1.15.}.


\begin{proof}[Proof of Theorem \ref{Lemma 1.15.}] 
Let $p_1,\dots, p_n$ be the prime divisors of $F(G)$, so that $F(G) = O_{p_1} (G) \times \cdots \times O_{p_n} (G)$. Let $H_i$ be a Hall $p_i$-complement of $G$ for each $i = 1, \dots, n$ and let $C_i = C_{H_i} (O_{p_i} (G))$. By Theorem \ref{Lemma 1.14.}, we know that $G/C_i$ is one of: (1) a $p_i$-group; (2) a Frobenius group with Frobenius complement $O_{p_i}(G)C_i/C_i$; (3) $p_i = 2$ and $G/C_i$ is isomorphic to $S_4$, ${\rm SL}_2 (3)$, or $\widetilde {{\rm GL}_2 (3)}$; or (4) $p_i = 3$, $G/C_i$ is a Frobenius group whose Frobenius kernel is the Sylow $3$-subgroup of $G/C_i$, a Frobenius complement is cyclic of order $2$, and $G/O_3 (G) C_i \cong S_3$. Note that if $G/(O_{p_i} (G) C_i)$ is isomorphic to a Frobenius complement, then it has normal subgroup that is metacyclic and whose quotient is isomorphic to a subgroup of $S_4$. It is not hard to show that $G/(O_{p_i} (G) C_i)$ has Fitting height at most $3$ and derived length at most $4$ in all cases.  If $|G|$ is odd, then $G/(O_{p_i} (G) C_i)$ is either cyclic or metacyclic; so it has Fitting height and derived length at most $2$. 

Take $N = \cap_{ i=1}^n O_{p_i} (G)C_i$.  It follows that $G/N$ has Fitting height at most $3$ and derived length at most $4$ and if $|G|$ is odd, then $G/N$ has Fitting height and derived length at most $2$.  Note that if $n = 1$, then $C_G (O_{p_1} (G)) = C_G (F(G)) \leq F(G)$, and so, $C_1 = 1$. This implies that $N = O_{p_1}(G) C_1 = O_{p_1} (G) = F (G)$.  Suppose that $n \ge 2$. Consider an element $x \in N$. We can write $x = x_1 \cdots x_n x'$ where $x_1, \dots, x_n, x'$ are powers of $x$ and each $x_i$ has $p_i$-power order and $x'$ has $\{ p_1, \dots, p_n\}'$-order. Since $x \in O_{p_i} (G) C_i$, we see that $x_i \in O_{p_i} (G)$ for each $i$. Also, $x' \in C_i$ for each $i$. This implies that $x'$ centralizes $O_{p_i} (G)$ for each $i$. We deduce that $x'$ centralizes $F(G)$, and so, $x' \in C_G (F(G)) \le F (G)$.  Since the order of $x'$ is coprime to $|F(G)|$, we see that $x' = 1$. Now, $x = x_1 \cdots x_n \in O_{p_1} (G)\times \cdots  \times O_{p_n} (G) = F(G)$.  We conclude that $N \le F(G)$. We have now proved the result.
\end{proof}

\end{document}